\documentclass[11pt]{article}

\usepackage{fullpage}
\usepackage{amsmath, amsthm, amsfonts, amssymb, amstext, mathrsfs, enumerate}
\usepackage{graphicx, ragged2e, lscape, framed, xcolor}
\usepackage{subfiles}

\theoremstyle{plain}
\newtheorem{theorem}{Theorem}[section]
\newtheorem{lemma}[theorem]{Lemma}

\newtheorem{problem}[theorem]{Problem}
\newtheorem{claim}{Claim}[section]
\newtheorem*{remark}{Remark}

\numberwithin{equation}{section}
\allowdisplaybreaks

\newcommand{\affl}[3]{\noindent #1, Email: {\tt #2}\\ \textsc{#3}\\[1.5pt]}

\usepackage[pagebackref]{hyperref}
\hypersetup{
	colorlinks=true,
    urlcolor=purple,
	linkcolor=purple,
    citecolor=purple,
}

\DeclareMathOperator{\diag}{diag}
\DeclareMathOperator{\rank}{rank}
\DeclareMathOperator{\tr}{tr}

\DeclareMathOperator{\Ger}{Ger}
\DeclareMathOperator{\sign}{sgn}
\newcommand{\one}{\mathbf{1}}
\newcommand{\ip}[2]{\langle #1, #2\rangle}

\def\b{\mbox{\boldmath $b$}}
\def\u{\mbox{\boldmath $u$}}
\def\v{\mbox{\boldmath $v$}}
\def\w{\mbox{\boldmath $w$}}
\def\d{\mbox{\boldmath $d$}}
\def\y{\mbox{\boldmath $y$}}
\def\x{\mbox{\boldmath $x$}}
\def\1{\mbox{\boldmath $1$}}
\def\0{\mbox{\boldmath $0$}}

\title{\textbf{Extremal graphs for the $k$-th eigenvalue}}
\author{Hitesh Kumar, Bojan Mohar, Seyed Ahmad Mojallal, Shivaramakrishna Pragada}
\date{}

\begin{document}
\maketitle
\begin{abstract}
For a simple graph $G$ of order $n$, let $\lambda_1(G)\ge \cdots \ge \lambda_n(G)$ denote its adjacency eigenvalues. Hong's problem asks for the optimal upper bound for $\lambda_k(G)$. A recent theorem of Sivashankar
gives, for every $k\ge3$,
\[
   \lambda_k(G)\le
   \frac{(k-2)\sqrt{k+1}+2}{2k(k-1)}\,n-1,
\]
with sharp examples arising from maximal real equiangular tight frames.
In this paper, we characterize the equality case. We also obtain an explicit combinatorial description of the extremal graphs for $\lambda_3$ and $\lambda_4$. 
\end{abstract}

\noindent
\textbf{Keywords:} 
\noindent Hong's problem, $k$-th largest eigenvalue, Equiangular tight frames, Gerzon graphs, Projection constants, Orthogonal projections, Seidel matrices.


\section{Introduction}
\subsection{Hong's problem}

Given a graph $G = (V(G), E(G))$ on $n$ vertices with \emph{adjacency matrix} $A(G)$, let
\[\lambda_1(G)\ge \lambda_2(G)\ge \cdots \ge \lambda_n(G)\]
denote the eigenvalues of $A(G)$. Let $\mathcal{G}(n)$ denote the set of all graphs of order $n$. 

In the spectral graph theory literature, $\lambda_1$, $\lambda_2$, and $\lambda_n$ have been extensively investigated, but much less is known about the behavior of intermediate $k$-th eigenvalue $\lambda_k$. In 1989, Powers \cite{Powers_1989} proposed the following elegant result for the $k$-th largest eigenvalue of graphs: for any fixed $k$ and any $G\in \mathcal{G}(n)$ $(n\ge k)$, we have 
\begin{equation}\label{eq:n_by_k_bound}
   \lambda_k(G)\le \frac{n}{k}-1. 
\end{equation}
If true, this bound would be tight for the graph obtained by taking a disjoint union of cliques of order $\frac{n}{k}$. Unfortunately, as pointed out by Nikiforov \cite{Nikiforov_2015}, Powers' proof was flawed. Independent of Powers' work, Hong \cite{Hong_1993} raised the following problem in 1993 and solved it for $k\le 2$ \cite{Hong_1988}.  

\begin{problem}[Hong's problem \cite{Hong_1993}] Find tight upper bounds for the $k$-th eigenvalue $\lambda_k(G)$ for $G\in \mathcal{G}(n)$ in terms of order $n$.    
\end{problem}

Hong's problem remained unattended until 2015, when Nikiforov \cite{Nikiforov_2015} put it in a systematic framework. In Nikiforov's words,
\begin{quote}
    ``... the general problem of Hong has never been tackled seriously. This is all the more inexplicable, as the problem is indeed challenging, easier for some values of $k$, and well beyond reach for others. What is more, Hong’s problem is not a backyard puzzle that is of interest only to spectral graph theorists; it is related to other fundamental areas of combinatorics and analysis, like the existence of symmetric Hadamard matrices, Ramsey theory, and extremal norms of graphs. We feel that the appeal and the importance of Hong’s problem should attract more research."
\end{quote}

Since then, several papers have appeared tackling Hong's problem, which include some recent breakthroughs. Following Nikiforov \cite{Nikiforov_2015}, define 
\[ c_k:= \sup\left\{\frac{\lambda_k(G)}{n}:\ G\in \mathcal{G}(n),\ n\ge k\right\}.\]
It is known that $c_1 = 1$ and $c_2 = \frac{1}{2}$. For $k\ge 5$, Nikiforov \cite{Nikiforov_2015} proved that $c_k > \frac{1}{k}$, thus falsifying \eqref{eq:n_by_k_bound}. He asked if $c_3 = \frac{1}{3}$ and $c_4 = \frac{1}{4}$, i.e., is \eqref{eq:n_by_k_bound} valid for $k = 3,4$? In 2023, Linz \cite{Linz_2023} proved that $c_4 > \frac{1}{4}$. Recently, Tang \cite{Tang_2026} established that $c_3 = \frac{1}{3}$. Following that, Sivashankar \cite{Sivashankar_2026} established the following bound for $\lambda_k$.

\begin{theorem}[\cite{Sivashankar_2026}] \label{thm:k_eigenvalue_bound_sivashankar}
Let $k\ge 3$. For any graph $G$ of order $n$, we have 
 \[ \lambda_k(G)\le \frac{(k-2)\sqrt{k+1}+2}{2k(k-1)}n-1.\]
\end{theorem}
Sivashankar \cite{Sivashankar_2026} also provided tight examples for the above bound for $\lambda_k$ when $k\in \{3, 4, 8, 24\}$ arising from the theory of equiangular lines; thus, determining the new values 
\[c_4 = \frac{\sqrt{5}+1}{12},\quad c_8 = \frac{5}{28}, \quad c_{24} = \frac{7}{69}.\]

\subsection{Connection to projection constants}

A matrix $Q = [q_{ij}]\in \mathbb{R}^{n\times n}$ is called an \emph{orthogonal projection} if 
\[Q = Q^2 = Q^\top.\]
The $1$-norm of a matrix $Q$ is given by 
\[\|Q \|_1:= \sum_{i,j} |q_{ij}|.\]
Tang \cite{Tang_2026} and Sivashankar \cite{Sivashankar_2026} both use an inequality for the $1$-norm of an orthogonal projection (see Theorem \ref{thm:Q_1_norm}) in their proofs. Such inequalities are well-studied in the theory of projection constants. 

For integers $n\ge r$, consider the set $\mathcal{P}_r(n)$ of orthogonal projections of order $n$ and rank $r$, i.e.,
\[
 \mathcal{P}_r(n):=
 \left\{Q\in\mathbb{R}^{n\times n}: Q^\top=Q=Q^2,\  \rank(Q)=r\right\}.
\]
Define
\[
 \mu_{\mathbb{R}}(r,n):=
 \frac{1}{n}\max_{Q\in\mathcal{P}_r(n)}\|Q\|_1,
 \qquad
 \mu_{\mathbb{R}}(r):=
 \sup_{n\ge r}\mu_{\mathbb{R}}(r,n).
\]
The quantities $\mu_{\mathbb{R}}(r,n)$ and $\mu_{\mathbb{R}}(r)$ are known as
the \emph{quasimaximal relative} and \emph{quasimaximal absolute projection constants}, respectively.

Let $\ell_\infty^n(\mathbb{R})$ denote the vector space
$\mathbb{R}^n$ equipped with the max norm 
\[\|\x\|_\infty:=\max_{i}|x_i|,\quad \x \in\mathbb{R}^n.\]  
For a finite-dimensional subspace $Y\subseteq \ell$, let
\[
 \lambda_{\mathbb{R}}(Y,\ell_\infty^n(\mathbb{R})):=
 \inf\left\{\|P\|:\ P:\ell_\infty^n(\mathbb{R}))\longrightarrow Y
 \text{ is a projection}\right\}.
\]

The \emph{maximal relative} and \emph{maximal absolute projection constants} are defined as
\[\lambda_{\mathbb{R}}(r,n):= \sup_{\substack{Y\subseteq\ell_\infty^n(\mathbb{R})\\ \dim Y=r}}
 \lambda_{\mathbb{R}}(Y,\ell_\infty^n(\mathbb{R})),
 \qquad
 \lambda_{\mathbb{R}}(r):=
 \sup_{n\ge r}\lambda_{\mathbb{R}}(r,n),
\]
respectively, for more details regarding projection constants and work associated with them, see \cite{Deregowska_Lewandowska_2023, Foucart_Skryzpek_2017, Wakhare_2026}.

For fixed $n$, it is known that
\[\mu_{\mathbb{R}}(r,n)\le \lambda_{\mathbb{R}}(r,n),\]
while Basso
\cite{Basso_2019} proved that their absolute versions coincide. Deregowska and
Lewandowska \cite{Deregowska_Lewandowska_2023} subsequently gave another proof
of this identity:
\begin{equation}\label{eq:projection_constant_identity}
 \mu_{\mathbb{R}}(r)=\lambda_{\mathbb{R}}(r).
\end{equation}
The relevance of \eqref{eq:projection_constant_identity} to graph eigenvalues
was made explicit by Wakhare \cite{Wakhare_2026}. Indeed, he showed that, for any graph $G$, the following inequality holds:
\begin{equation}\label{eq:projection_constant_graph_eigenvalue}
    \lambda_{r+1}(G) \leq \frac{\lambda_{\mathbb{R}}(r)}{2r}n -1.
\end{equation}
Deregowska and Lewandowska \cite{Deregowska_Lewandowska_2023} also proved the
general estimate
\begin{equation}\label{eq:projection_constant_upper_bound}
 \lambda_{\mathbb{R}}(r)
 \le
 \frac{2}{r+1}
 \left(1+\frac{r-1}{2}\sqrt{r+2}\right)
 =
 \frac{r+\sqrt{r+2}}{1+\sqrt{r+2}}
 =:\beta_r.
\end{equation}
Moreover, equality holds in \eqref{eq:projection_constant_upper_bound} whenever
there exists a maximal real equiangular tight frame consisting of
$\binom{r+1}{2}$ vectors in $\mathbb{R}^r$. The estimate for orthogonal
projections used by Sivashankar is therefore an immediate consequence of the
projection-constant bound. The proof given by Sivashankar \cite{Sivashankar_2026} uses positive semidefinite functions on the sphere and their connection to Gegenbauer polynomials, whereas Deregowska and Lewandowska \cite{Deregowska_Lewandowska_2023} derive the same bound through projection constants and show that it is attained whenever a maximal real equiangular tight frame exists.

\begin{theorem}[\cite{Deregowska_Lewandowska_2023,Sivashankar_2026}]
\label{thm:Q_1_norm_sivashankar}
Let $Q\in\mathbb{R}^{n\times n}$ be a rank-$r$ orthogonal projection, where
$r\ge 2$. Then
\[
 \|Q\|_1\le \beta_r n,
 \qquad
 \beta_r=\frac{r+\sqrt{r+2}}{1+\sqrt{r+2}}.
\]
\end{theorem}

The case $r=2$ gives $\lambda_{\mathbb{R}}(2)=4/3$. This was conjectured by Grünbaum \cite{Grunbaum_1960}, proved by Chalmers and Lewicki \cite{Chalmers_Lewicki_2010}, and reproved by Basso \cite{Basso_2019} using a different method. Thus, the rank two estimate underlying Tang's proof using the trigonometric majorants is another derivation of the upper bound in Grünbaum's conjecture \cite{Tang_2026, Wakhare_2026}. 

Deregowska and Lewandowska also show that equality in the above theorem is attained if there exists a maximal equiangular tight frame consisting of
$\binom{r+1}{2}$ vectors in $\mathbb{R}^r$. 

\subsection{Equiangular tight frames (ETF) and Gerzon graphs} \label{subsection:ETF-Gerzon-graphs}

The equality cases for \ref{thm:k_eigenvalue_bound_sivashankar} and for the inequalities involving projection constants arise from equiangular tight frames (ETFs). Here we give some background on ETFs that is necessary for further development in our paper. 

Suppose there is a set of unit vectors $\mathcal{V} = \{
\v_1,\ldots,\v_N\}\subset \mathbb{R}^r$ $(r\ge 1)$ satisfying
\[ \sum_{i=1}^N \v_i \v_i^\top = \frac{N}{r} I_r\qquad \text{and}\qquad |\ip{\v_i}{\v_j}|= \sqrt{\frac{N - r}{r(N-1)}} \quad (i\neq j). \]
Then $\mathcal{V}$ is said to form an \emph{equiangular tight frame} (ETF); see \cite{MR2149656, MR2350682, Waldron_2009} for more details. If
\[ N: = \frac{r(r+1)}{2},\]
we say that $\mathcal{V}$ forms a \emph{Gerzon equiangular tight frame} (Gerzon ETF) (also known as \emph{maximal} ETF in literature) with common angle $\frac1{\sqrt{r+2}}$. 

\textbf{Construction of Gerzon graphs of dimension $r$.}

Consider a multi-set $X = \{\x_1, \ldots, \x_n\}\subset \mathbb{R}^r$ of $n$ unit vectors. Denote by $G_{X}$ the simple graph with vertex set
$V(G_{X})=X$, and for distinct $i,j\in[n]$,
\[
    \x_i \sim \x_j \iff    \langle\x_i,\x_j\rangle>0.
\]
Let $n = Nt$ for some positive integer $t$ and $\mathcal{V} =\{\v_1, \ldots, \v_N\}$ be a Gerzon ETF. Consider a multi-set $\mathcal{U} = \{\u_1, \ldots, \u_n\}$ such that $u_i\in \{\pm \v_1, \ldots, \pm \v_N\}$ for all $i\in [n]$. Let $V$ denote the $r\times N$ matrix with columns $\v_1, \ldots, \v_N$. Define 
\[p_i:=|\{\u_j\in \mathcal{U}: \u_j = \v_i\}|,\quad  n_i:=|\{\u_j\in \mathcal{U}: \u_j = -\v_i\}|, \quad d_i:= p_i - n_i, \quad \d := (d_i) \in \mathbb{Z}^N.\]
We say that $\mathcal{U}$ is \emph{feasible} if the following conditions hold:
\begin{equation}\label{eq:feasible_U}
    V\d = \0 \quad \text{and}\quad p_i+n_i = t \quad (i\in [n]).  
\end{equation}
Also, observe here that
\[ p_i + n_i = t \quad \iff \quad  p_i = \frac{t+d_i}{2}\quad \text{and}\quad n_i = \frac{t-d_i}{2},\]
and 
\[ V\d = \0 \quad \iff \quad \sum_{i=1}^n \u_i = \0.\]
We say that $G_{\mathcal U}$ is a \emph{Gerzon graph}
associated with $\mathcal U$ whenever $\mathcal{U}$ is feasible. 

Define 
\[\mathcal{V}_* := \mathcal{V}\cup (-\mathcal{V}) \quad \text{ where }\quad -\mathcal{V} := \{-\v_1, \ldots, -\v_N\}.\]
It is clear that $\mathcal{V}_*$ satisfies the two conditions \ref{eq:feasible_U} and is therefore feasible.

\begin{lemma} For any Gerzon ETF $\mathcal{V}$, the graph $G_{\mathcal{V}_*}$ is a Gerzon graph.
\end{lemma}

Notice that $G_\mathcal{V}\cong G_{\mathcal{-V}}$, and both are induced subgraphs of $G_{\mathcal{V}_*}$.

Moreover, for any $\mathcal{U}$ (not necessarily feasible), the graph $G_\mathcal{U}$ can be obtained from $G_{\mathcal{V}_*}$ as follows: for every $i\in\{1,\ldots,N\}$, replace vertex $\v_i\in V(G_{\mathcal{V}_*})$ with a clique $C_i^+$ of order $p_i$ and replace every vertex $-\v_i\in V(G_{\mathcal{V}_*})$ with a clique $C_i^-$ of order $n_i$. Moreover, for $i\ne j$ and $\varepsilon,\delta\in\{\pm1\}$, every vertex of $C_i^\varepsilon$ is joined to every vertex of $C_j^\delta$ in $G_{\mathcal{U}}$ if and only if $\varepsilon \v_i\sim \delta \v_j$ in $G_{\mathcal{V}_*}$. And there are no edges between $C_i^+$ and $C_i^-$ in $G_{\mathcal{U}}$. We say that $G_{\mathcal{U}}$ is an \emph{(uneven) closed blowup} of $G_{\mathcal{V}_*}$. We call $G_{\mathcal{V}_*}$ the \emph{base graph}. Note that $p_i$ and $n_i$ can be zero.

We denote by $\Ger_{\mathcal{V}}(r,n)$ the family of all Gerzon graphs $G_{\mathcal{U}}$, considered up to isomorphism, that arise from taking feasible multi-sets $\mathcal{U}$ whose elements come from Gerzon ETF $\mathcal{V}$. As we shall see later, the problem of characterizing the tight cases for the inequality in Theorem \ref{thm:k_eigenvalue_bound_sivashankar} is equivalent to determining $\Ger_{\mathcal{V}}(r,n)$. 

It is apriori not clear that $\Ger_{\mathcal{V}}(r, n)$ is equal to $\Ger_{\mathcal{V}'}(r, n)$ where $\mathcal{V}'$ is obtained from $\mathcal{V}$ via sign changes. For completeness, we show that this is indeed the case.

\begin{lemma}\label{lem:switching_invariance}
Consider a Gerzon ETF $\mathcal{V} = \{\v_1, \ldots, \v_N\}\subset \mathbb{R}^r$ and another Gerzon ETF $\mathcal{V}' = \{\varepsilon_1\v_1, \ldots, \varepsilon_N\v_N\}$ where $\varepsilon_i\in \{\pm 1\}$. Then $\Ger_{\mathcal{V}}(r, n) = \Ger_{\mathcal{V}'}(r, n)$. 
\end{lemma}

\begin{proof} To prove the assertion, we only need to show that if $\mathcal{U}$ is a feasible multi-set arising from $\mathcal{V}$, then $\mathcal{U}$ is feasible for $\mathcal{V}'$ as well. 

First, since $\mathcal{U}$ is fixed, the quantity $p_i + n_i$ remains fixed for any $i\in [n]$ after sign changes from $\mathcal{V}$ to $\mathcal{V}'$. 

Denote by $\d$ and $\d'$ the vectors in $\mathbb{Z}^N$ associated with $\mathcal{U}$ corresponding to $\mathcal{V}$ and $\mathcal{V}'$, respectively. Let $D = \diag(\varepsilon_1, \ldots, \varepsilon_N)$. Denote by $V$ the $r\times N$ matrix with columns $\v_1, \ldots, \v_N$, and denote by $V'$ the $r\times N$ matrix with columns $\varepsilon_1\v_1, \ldots, \varepsilon_N\v_N$. Clearly, 
\[ V' = VD \quad \text{and}\quad \d' = D\d.\]
Thus, 
\[ V'\d' = V\d = 0\]
since $D^2 = I$. Therefore, $\mathcal{U}$ satisfies \eqref{eq:feasible_U} w.r.t. $\mathcal{V}'$. The assertion follows.
\end{proof}

\textbf{Connection to Seidel graphs.}

Note that $V^\top V$ is the \emph{gram matrix} of the vectors $\v_1,\ldots,\v_N$. Moreover,
\begin{equation}\label{eq:Seidel_1}
    V^\top V
    =
    I+\frac{1}{\sqrt{r+2}}S(\Gamma),
\end{equation}
where $S(\Gamma)=J-I-2A(\Gamma)$ is the \emph{Seidel matrix} of a graph
$\Gamma$. In particular, for distinct $i,j$,
\[
    ij\in E(\Gamma) \iff
    S(\Gamma)_{ij}=-1
     \iff
    \langle \v_i,\v_j\rangle
    =-\frac{1}{\sqrt{r+2}}.
\]
By \eqref{eq:Seidel_1}, it is easy to see that
\begin{equation}\label{eq:Seidel_2}
   V\d = \0 \iff S(\Gamma)\,\d = -\sqrt{r+2}\, \d,   
\end{equation}
i.e., $\d$ is an eigenvector of $S(\Gamma)$ corresponding to eigenvalue $-\sqrt{r+2}$ whenever $\d \neq \0$. Also, note that $\d = \0$ and any even positive integer $t$ always satisfy the above two conditions. Thus, in order to determine $\Ger_{\mathcal{V}}(r,n)$ for given Gerzon ETF $\mathcal{V}$, we need only determine all possible integer eigenvectors $\d$ corresponding to Seidel graph $\Gamma$ with eigenvalue $-\sqrt{r+2}$ such that 
\[ |d_i|\le t \quad \text{and}\quad d_i \equiv t \mod 2 \quad (1\le i\le N).\]

\subsection{Our contribution}

The characterization of the equality case for the upper bound for $\lambda_k$ in Theorem \ref{thm:k_eigenvalue_bound_sivashankar} was left as an open problem. Our main contribution is to characterize the equality. Sivashankar \cite{Sivashankar_2026} asked if all the tight examples arise from equiangular lines, and we confirm this to be true. We show the following.

\begin{theorem}\label{thm:k_eigenvalue_bound} Let $k\ge 3$. For any graph $G$ of order $n$, we have 
 \[ \lambda_k(G)\le \alpha_k n - 1 \qquad \text{where}\qquad \alpha_k :=\frac{(k-2)\sqrt{k+1}+2}{2k(k-1)}.\]
Equality holds if and only if $G\in \Ger_{\mathcal{V}}(r,n)$, where $r = k-1$, $n =t \frac{r(r+1)}{2}$ for some $t\in \mathbb{Z}^+$ and $\mathcal{V}$ is a Gerzon ETF in $\mathbb{R}^r$.
\end{theorem}

Consequently, the bound for $\lambda_k$ is attained for a given $k\ge3$ if and
only if there exists a Gerzon ETF in
$\mathbb R^{k-1}$. Hence, tight cases are known for $k\in\{3,4,8,24\}$. For all remaining values of $k$ for which the Gerzon bound is known not to be attainable, the inequality is necessarily strict for every graph. Determining the optimal strict improvement in those cases remains an
interesting problem. We prove Theorem \ref{thm:k_eigenvalue_bound} in Section \ref{section:Bounds}.

To prove Theorem \ref{thm:k_eigenvalue_bound}, we also characterize the equality case for the upper bound for the 1-norm of orthogonal projections $Q$ of rank $r$ in Theorem \ref{thm:Q_1_norm}. This is proved in Section \ref{section:Bounds}.

Finally, in the special case when $k\in\{3,4\}$, we give a complete description of the family of graphs $\Ger_{\mathcal{V}}(k-1, n)$ which are precisely the tight examples for the upper bound for $\lambda_k$. 

To state our results for $\lambda_3$ and $\lambda_4$, we need to define some notation. If $H$ is a graph and $(m_v)_{v\in V(H)}$ are nonnegative integers, the \emph{(uneven) closed blowup} of $H$ with \emph{blowup multiplicities} $(m_v)$ is the graph obtained by replacing each vertex $v$ by a clique $C_v$ of size $m_v$, and by joining $C_u$ completely to $C_v$ whenever $uv\in E(H)$. If $m_v = q$ for all $v\in V(H)$, then it is called the \emph{uniform closed $q$-blowup} of $H$ and is denoted $H^{[q]}$.

Define $H_{a,b}$ to be the closed blowup of the $6$-cycle $C_6$ with multiplicities $a,b,a,b,a,b$ in a cyclic order where $a,b\ge 0$ and $a+b = \frac{|H_{a,b}|}{3}$. Leonida-Li \cite{Leonida_Li_2026} showed that $\lambda_3(H_{a,b}) = \frac{|H_{a,b}|}{3}-1$. We show the following.

\begin{theorem}\label{thm:3_eigenvalue}
For any graph $G$ of order $n\ge 3$, we have 
 \[ \lambda_3(G)\le \frac{n}{3} - 1.\]
Equality holds if and only if $n = 3(a+b)$ and $G\cong H_{a,b}$.
\end{theorem}

We denote by $\mathcal{I}_{12}$ the icosahedral graph on $12$ vertices. The closed blowups of the icosahedral graph were used by Linz \cite{Linz_2023} to give a lower bound for maximum $\lambda_4$. We show that these are precisely the extremal graphs for $\lambda_4$.  

\begin{theorem}\label{thm:4_eigenvalue} 
For any graph $G$ of order $n\ge 4$, we have 
 \[ \lambda_4(G)\le \frac{1+ \sqrt{5}}{12}n - 1.\]
Equality holds if and only if $n = 12q$ and $G\cong \mathcal{I}_{12}^{[q]}$ where $q\in \mathbb{N}$.
\end{theorem}

These results are discussed in Section \ref{section:special_k}. We conclude Section \ref{section:special_k} with some combinatorial description of the extremal graphs for $\lambda_8$.

\section{Bounds and equality}\label{section:Bounds}

In this section, our main goal is to prove Theorem \ref{thm:k_eigenvalue_bound}. We begin with some necessary background on positive semidefinite functions.

\subsection{Gegenbauer polynomials and positive semidefinite functions}

Let $S^{r-1}$ denote the unit sphere in $\mathbb{R}^r$. For real vectors $\x,\y$, let $\ip{\x}{\y}$ denote the standard inner product. A function $f:\mathbb{R}\rightarrow\mathbb{R}$ is \emph{positive semidefinite} on $S^{r-1}$ if for all unit vectors $\u_1,\ldots,\u_n \in S^{r-1}$, the matrix $M \in \mathbb{R}^{n \times n}$ given by $M_{ij} = f(\ip{\u_i}{\u_j})$ is positive semidefinite (PSD).

Let $G^\lambda_\ell$ denote the Gegenbauer polynomial of degree $\ell$. These polynomials can be defined in terms of their generating function \cite{Stein_Weiss_1971}:
\[\frac{1}{(1-2tx+x^2)^\lambda} = \sum_{\ell =0}^\infty G^\lambda_\ell(t)x^\ell\]

Schoenberg \cite{Schoenberg_1942} characterized positive semidefinite continuous functions on $S^{r-1}$ in terms of Gegenbauer polynomials. Indeed he proved that a continuous function $f(t)$ is positive semidefinite on $S^{r-1}$ if and only if $f$ is of the form $\sum_{\ell = 0}^\infty a_\ell G^{r/2-1}_\ell$ with $a_\ell \geq 0$. 

We need a special case of the above theorem for the following two polynomials:
\[f_2^r(t) := t^2 - \frac{1}{r} \quad \text{ and } \quad f_4^r(t) := t^4 - \frac{3}{r(r+2)}.\]
It can be shown that
\[f^r_2(t) = \frac{r-1}{r} G^{r/2-1}_{2}(t)\quad \text{ and }\quad f^r_4(t) = \frac{(r-1)(r+1)}{(r+2)(r+4)} G^{r/2-1}_{4}(t) + \frac{6(r-1)}{r(r+4)} G_2^{r/2-1}(t),\]
and hence by Schoenberg's result both $f^r_2$ and $f^r_4$ are positive semidefinite on $S^{r-1}$. 
\begin{lemma}[cf. \cite{Sivashankar_2026}] \label{lemma:f_r_psd}
    For $r \ge 2$, the polynomials $f_2^r$ and $f_4^r$ are positive semidefinite on $S^{r-1}$.
\end{lemma}

In the next lemma, we estimate $|t|$ in terms of polynomials $f^r_2(t)$ and $f^r_4(t)$. The inequality \eqref{eq:f_r_polynomial_inequality} was established 
in \cite{Sivashankar_2026}, but here we also describe the equality case.

\begin{lemma}\label{lemma:f_r_polynomial_inequality}
Let $r \ge 2$ be an integer. Define
\[s := \sqrt{r+2}, \qquad a_r := \frac{s^2+s-2}{r(s+2)}, \qquad b_r := \frac{s(s^2+2s+3)}{2(s+1)^2}, \qquad \gamma_r:=\frac{s^3}{2(s+1)^2}.\]
Then, we have $b_r < ra_r$, and for every $t \in [-1,1]$,
\begin{equation}\label{eq:f_r_polynomial_inequality}
  |t| \le a_r + b_r f^r_2(t) -\gamma_rf^r_4(t).  
\end{equation}
The equality in \eqref{eq:f_r_polynomial_inequality} holds if and only if $|t| \in \{1,\frac{1}{\sqrt{r+2}}\}$. 
\end{lemma}
\begin{proof} A simple calculation shows that
\[ra_r -b_r = \frac{s^3+2s^2-5s-4}{2(s+1)^2} = \frac{(s-2)(s+1)(s+3)+2}{2(s+1)^2}>0,\]
implying $b_r < ra_r$. 

We now prove \eqref{eq:f_r_polynomial_inequality}. Since $|t|$, $f_2^r(t)$ and $f_4^r(t)$ are even functions in $t$, it suffices to consider $t \in [0,1]$. A direct calculation shows that
\[a_r + b_r f^r_2(t) -\gamma_rf^r_4(t) -t = \frac{(1-t)(st-1)^2(st+s+2)}{2(s+1)^2}.\]
The right-hand side is nonnegative for $t \in [0,1]$ since
\[(1 - t) \ge 0, \quad (st - 1)^2 \ge 0, \quad st + s + 2 > 0.\]
Thus, \eqref{eq:f_r_polynomial_inequality} holds. Clearly, equality can hold in \eqref{eq:f_r_polynomial_inequality} if and only if $|t|\in \{1, \frac{1}{s}\}$.
\end{proof}

\subsection{Rank $r$ projections}

Here, we estimate the 1-norm of rank $r$ orthogonal projections. The inequality in the following result was established by Sivashankar \cite{Sivashankar_2026}. We revisit his proof and also characterize the equality case.

\begin{theorem}\label{thm:Q_1_norm}
    Let $Q\in \mathbb{R}^{n \times n}$ be a rank $r$ orthogonal projection with $r\ge 2$. Then 
    \[ \|Q \|_1 \le \beta_r n \qquad \text{where} \qquad \beta_r:=\left(\frac{r+\sqrt{r+2}}{1+\sqrt{r+2}}\right).\]
    The equality holds if and only if \[Q = S^\top \left(\frac{r}{n}V^\top V\otimes J_{\frac{n}{N}}\right)S,\] 
     where $V$ is a $r\times N$ matrix whose columns $\v_1, \ldots, \v_N\in \mathbb{R}^r$ form a Gerzon equiangular tight frame with common angle $\frac{1}{\sqrt{r+2}}$. Here, $J_{\frac{n}{N}}$ is the all-ones matrix of order $\frac{n}{N}$, $S$ is a signed permutation matrix, and $\otimes$ denotes the Kronecker product.
\end{theorem}

\begin{proof} Without loss of generality, we can assume that $Q$ does not have any zero column or row. Since $Q$ is a rank $r$ orthogonal projection, we can write $Q = BB^\top$ where $B\in \mathbb{R}^{n\times r}$ is such that the columns of $B$ form an orthonormal set in $\mathbb{R}^n$, i.e., $B^\top B = I_r$. Let $\b_1, \ldots, \b_n$ denote the rows of $B$. Since $Q$ does not have any zero column, we see that all $\b_i$'s are non-zero vectors. Let 
\[c_i:=\|\b_i\|_2>0 \qquad \text{and}\qquad \u_i:=\frac{\b_i}{c_i}.\] 
It is clear that 
    \begin{equation}\label{eq:setup}
       q_{ij} = c_ic_j\ip{\u_i}{\u_j}, \qquad\sum_{i=1}^n c_i^2 \u_i \u_i^\top = I_r, \qquad   \sum_{i=1}^n c_i^2=r. 
    \end{equation}
Define
\[  C:=\sum_{i=1}^n c_i, \qquad X:=\left\|\sum_{i=1}^n c_i\left(\u_i\u_i^\top-\frac1r I_r\right)\right\|_F.\]
Note that $\ip{\u_i}{\u_j} \in [-1,1]$ since $\u_i$'s are unit vectors. Thus, for any pair $i,j$, using Lemma \ref{lemma:f_r_polynomial_inequality} with $t = \ip{\u_i}{\u_j}$, we get 
    \begin{equation}\label{eq:q_i_j_bound_1}
        |q_{ij}|\le a_r c_ic_j + b_r c_ic_jf_2^r(\ip{\u_i}{\u_j}) - \gamma_r c_ic_j f_4^r(\ip{\u_i}{\u_j}).
    \end{equation}
    Summing over all pairs $i,j$, we get
    \begin{equation}\label{eq:q_i_j_bound_2}
        \|Q\|_1 = \sum_{i,j} |q_{ij}| \le a_r \sum_{i,j}c_ic_j + b_r\sum_{i,j} c_ic_jf_2^r(\ip{\u_i}{\u_j}) - \gamma_r \sum_{i,j} c_ic_j f_4^r(\ip{\u_i}{\u_j}).
    \end{equation}
    Observe that 
     \[ \sum_{i,j}c_ic_j = C^2,\]
     and 
     \begin{align*}
        \sum_{i,j} c_ic_jf_2^r(\ip{\u_i}{\u_j}) &= \sum_{i,j} c_ic_j \left(\ip{\u_i}{\u_j}^2 - \frac{1}{r}\right)\\
        &= \sum_{i,j} c_ic_j \tr\left(\left(\u_i\u_i^\top - \frac{1}{r}I_r\right)\left(\u_j\u_j^\top - \frac{1}{r}I_r\right) \right) \\
        & = \tr\left(\left( \sum_{i=1}^n c_i\left(\u_i\u_i^\top-\frac1r I_r\right)\right)^2\right)\\
        & = X^2.
     \end{align*}
     Moreover, by Lemma \ref{lemma:f_r_psd}, we have $\sum_{i,j} c_ic_jf_4^r(\ip{\u_i}{\u_j})\ge 0$. Thus,
     \[ \|Q\|_1 \le a_r C^2 + b_r X^2.\]
     Since $b_r < ra_r $ by Lemma \ref{lemma:f_r_polynomial_inequality}, we conclude
     \begin{equation}\label{eq:C_X_bound_1}
         \|Q\|_1 \le a_r \left(C^2 + rX^2\right).
     \end{equation}
     
     We now estimate $C^2 + rX^2$. So observe that 
     \begin{align*}
         C^2 + rX^2 & = \sum_{i,j} c_ic_j + r\sum_{i,j} c_ic_j \left(\ip{\u_i}{\u_j}^2 - \frac{1}{r}\right)\\
         & = r\sum_{i,j} c_ic_j \ip{\u_i}{\u_j}^2 \\
         & \le r \sum_{i,j} \frac{c_i^2 + c_j^2}{2} \ip{\u_i}{\u_j}^2 \qquad (\text{by AM-GM inequality})\\
         & =  r \sum_{i} \sum_j c_j^2 \ip{\u_i}{\u_j}^2\\
         & = r \sum_i \sum_j \tr\left( \u_i\u_i^\top c_j^2 \u_j\u_j^\top\right)\\
         & = r \sum_i \tr\left( \u_i\u_i^\top \sum_j c_j^2 \u_j\u_j^\top\right)\\
         & = r \sum_i \tr\left( \u_i\u_i^\top\right)\qquad (\text{by \eqref{eq:setup}})\\
         & = rn.
    \end{align*}
     We conclude that 
     \begin{equation}\label{eq:Q_1_norm}
         \|Q\|_1 \le a_r rn = \left(\frac{r+\sqrt{r+2}}{1+\sqrt{r+2}}\right)n = \beta_r n.
     \end{equation}
     
We now discuss the equality case. 

\textbf{Necessity.} Suppose that equality holds in \eqref{eq:Q_1_norm}. Then equality must hold in each inequality above. In
particular, since $b_r<ra_r$ by Lemma \ref{lemma:f_r_polynomial_inequality}, equality in
\[a_rC^2+b_rX^2 \le a_r(C^2+rX^2)\]
forces $X=0$. Also, equality in
\[C^2+rX^2\le rn \]
gives $C^2=rn$. Since $\sum_{i=1}^n c_i^2=r$, the Cauchy-Schwarz
inequality gives
\[C^2=\left(\sum_{i=1}^n c_i\right)^2\le n\sum_{i=1}^n c_i^2=nr. \]
Thus, the equality case in the Cauchy-Schwarz inequality implies
\begin{equation}
  c_i=\sqrt{\frac rn}  
\end{equation}
for every $i$. In particular, using \eqref{eq:setup}, we obtain
\begin{equation}\label{eq:equality_tight}
    \sum_{i=1}^n \u_i\u_i^\top=\frac nr I_r .
\end{equation}
Now, equality in \eqref{eq:q_i_j_bound_1} forces that for all pairs $i,j$, 
\begin{equation}\label{eq:equality_angles}
    |\ip{\u_i}{\u_j}|
    \in
    \left\{1,\frac1{\sqrt{r+2}}\right\},
\end{equation}
by Lemma \ref{lemma:f_r_polynomial_inequality} and the fact that $c_i>0$ for all $i$.

Now, let $\v_1,\dots,\v_m\in \mathbb{R}^{r}$ be representative unit vectors of the distinct lines among the $\u_i$'s, and let
\[t_j:=|\{\u_i: \u_i = \pm \v_j \}|.\]
Thus \[\sum_{j=1}^m t_j=n.\]
By \eqref{eq:equality_angles}, distinct lines $\v_i$ and $\v_j$ satisfy
\begin{equation}\label{eq:equiangular_lines}
    |\ip{\v_i}{\v_j}| = \frac1{\sqrt{r+2}}, \qquad i\ne j.
\end{equation}
Moreover, \eqref{eq:equality_tight} becomes
\begin{equation}\label{eq:weighted_tight_lines}
    \sum_{j=1}^m t_j \v_j \v_j^\top = \frac nr I_r .
\end{equation}

Multiplying  by $\v_q^\top$ on left and $\v_q$ on right of  \eqref{eq:weighted_tight_lines} and using \eqref{eq:equiangular_lines} gives 
\begin{align*}
   \frac{n}{r} = \sum_{j=1}^m t_j \v_q^\top (\v_j \v_j^\top)\v_q = \sum_{j=1}^m t_j (\v_q^\top \v_j)^2 = \left(\sum_{j \neq q} \frac{t_j}{r+2}\right) + t_q = \frac{n-t_q}{r+2} + t_q,
\end{align*}
implying
\[t_q=\frac{2n}{r(r+1)} = \frac{n}{N} \]
for every $q$, where $N=\frac{r(r+1)}2$. Since $\sum_q t_q=n$, it follows that $m=N$. Hence the distinct lines $\v_1, \ldots, \v_m$
form an equiangular tight frame of $N=r(r+1)/2$ lines in $\mathbb R^r$,
with common absolute inner product $\frac1{\sqrt{r+2}}$ and each line occurs with the same multiplicity $n/N$, up to arbitrary sign changes. It is now clear that 
\[ Q = BB^\top = S^\top\left(\frac{r}{n}V^\top V\otimes J_{\frac{n}{N}}\right)S,\]
where $S$ is some signed permutation matrix.

\textbf{Sufficiency.} Suppose that 
\[ Q = S^\top\left(\frac{r}{n}V^\top V\otimes J_{\frac{n}{N}}\right)S,\]
where the columns of $V$ form a Gerzon equiangular tight frame in $\mathbb{R}^r$ with common angle $\frac{1}{\sqrt{r+2}}$. We show that $Q$ is in fact orthogonal projection matrix, indeed
\begin{align*}
    Q^2 &= S^\top\left(\frac{r}{n} V^\top V \otimes J_{\frac{n}{N}}\right)S S^\top \left(\frac{r}{n} V^\top V \otimes J_{\frac{n}{N}}\right)S \\
    &= S^\top\left(\frac{r^2}{n^2}V^\top VV^\top V\right) \otimes \left(\frac{n}{N}J_{\frac{n}{N}}\right)S \\
    &= S^\top \left(\frac{r^2}{n^2}\frac{N}{r}V^\top V\right) \otimes \left(\frac{n}{N}J_{\frac{n}{N}}\right) S \quad (\text{using} VV^\top = \frac{N}{r}I_r)\\
    & = Q.
\end{align*}

Next, we show that $\|Q\|_1 = \beta_r n$. Using $N= \frac{r(r+1)}{2}$, we see that
\begin{align*}
    \|Q\|_1 & = \frac{r}{n}\cdot \left(\frac{n}{N}\right)^2\cdot \sum_{i,j} |\ip{\v_i}{\v_j}| \\
    & = \frac{r}{n}\cdot \left(\frac{n}{N}\right)^2\cdot \left(N + \frac{N(N-1)}{\sqrt{r+2}}\right)\\
    & = \frac{rn}{N}\left( 1 + \frac{N-1}{\sqrt{r+2}}\right)\\
    &  = n \left( \frac{r+\sqrt{r+2}}{1 + \sqrt{r+2}}\right) \\
    & = \beta_r n.
\end{align*}
This completes the proof.
\end{proof}

\subsection{Upper bound for $k$-th eigenvalue}

Here, we prove our main result Theorem \ref{thm:k_eigenvalue_bound}.

\begin{proof}[Proof of Theorem \ref{thm:k_eigenvalue_bound}]
    Let $r:= k -1$. Note that 
    \[A(G) + A(\overline{G}) = J - I\] where $\overline{G}$ is the complement of $G$, $J$ is the all-ones matrix and $I$ is the identity matrix. Using the well-known Weyl's inequality, we see that
    \[ \lambda_k(G) + \lambda_{n-k+2}(\overline{G})\le \lambda_2(J-I) = -1,\]
    which implies
    \begin{equation}\label{eq:complement_1}
       \lambda_k(G)\le -\lambda_{n-r+1}(\overline{G})-1. 
    \end{equation}
    So we only need to estimate $\lambda_{n-r+1}(\overline{G})$.

    To that end, let $\x_1, \ldots, \x_n$ denote the orthonormal set of eigenvectors corresponding to eigenvalues $\lambda_1(\overline{G}), \ldots, \lambda_n(\overline{G})$, respectively. Define 
    \[Q := \sum_{i=n-r+1}^{n} \x_i\x_i^\top.\]
    It is clear that $Q$ is a rank $r$ orthogonal projection with $\tr(Q) = r$. Now, 
    \begin{align*}
        r \lambda_{n-r+1}(\overline{G}) & \ge \sum_{i=n-r+1}^n \lambda_i(\overline{G}) \\
        & = \tr(A(\overline{G})Q)\\
        & = 2\sum_{i<j}A(\overline{G})_{ij} q_{ij}\qquad (\text{where }q_{ij}:=Q_{ij})\\
        & \ge 2\sum_{i<j}\min\{q_{ij}, 0\}\\
        & = \sum_{i<j} q_{ij} - \sum_{i<j}|q_{ij}|\\
        & = \left(\frac{\mathbf{1}^\top Q1 - \tr(Q)}{2}\right) - \left(\frac{\|Q\|_1 - \tr(Q)}{2} \right)\\
        & \ge -\frac{\|Q\|_1}{2} \qquad (\text{since }\mathbf{1}^\top Q\mathbf{1}\ge 0)\\
        & \ge - \frac{\beta_r n }{2} \qquad (\text{by Theorem \ref{thm:Q_1_norm}}).
    \end{align*}
    We conclude that 
    \begin{equation}\label{eq:complement_2}
      \lambda_{n-r+1}(\overline{G})\ge - \frac{\beta_r n }{2r} = -\alpha_k n.
    \end{equation}
    Using \eqref{eq:complement_1}, we get the desired inequality for $\lambda_k(G)$.

    Let us now analyze the equality case. 

    \begin{claim}[Necessity] Suppose $G$ is a graph such that $\lambda_k(G) = \alpha_k n - 1$. Then $G\in \Ger_{\mathcal{V}}(r,n)$ for some Gerzon ETF $\mathcal{V}$.
    \end{claim}

    \begin{proof}
        If $\lambda_k(G) = \alpha_k n - 1$, then all of the inequalities in the above proof must be equalities. So, using \eqref{eq:complement_1}, we have
        \[ \lambda_{n-r+1}(\overline{G}) = -\alpha_k n.\]
        Moreover,
        \[  r \lambda_{n-r+1}(\overline{G}) = \sum_{i=n-r+1}^n \lambda_i(\overline{G})\]
        forces that the bottom $r$ eigenvalues of $\overline{G}$ are equal, i.e.,
        \[ \lambda_{n-r+1}(\overline{G}) = \cdots = \lambda_n(\overline{G}) = - \alpha_k n.\]
        Since $\|Q\|_1 = \beta_r n$, by Theorem \ref{thm:Q_1_norm} we have that 
        \begin{equation}\label{eq:Q_expression_1}
          Q = S^\top\left(\frac{r}{n}V^\top V\otimes J_{\frac{n}{N}}\right)S  
        \end{equation}
        where $V$, $J_{\frac{n}{N}}$ and $S$ are as defined in Theorem \ref{thm:Q_1_norm}. Let $U$ denote a $r\times n$ matrix with columns $\sqrt{\frac{r}{n}}
        \u_1, \ldots, \sqrt{\frac{r}{n}}\u_n$ such that 
        \[U^\top U = Q.\] 
        For $1\le i\le N$, define 
        \[ p_i:=|\{\u_j: \u_j = \v_i\}| \quad \text{and}\quad n_i:=|\{\u_j: \u_j = -\v_i\}|.\]
        Clearly, $p_i+n_i = \frac{n}{N}$ for all $i$. Furthermore, since $\mathbf{1}^\top Q\mathbf{1} = 0$, we see that
        \[ 0 = \mathbf{1}^\top U^\top U\mathbf{1} = \|U\mathbf{1}\|_2^2= \left\|\sum_{i=1}^n \u_i\right\|_2^2,\]
        which implies
        \[ \sum_{i=1}^n \u_i = \mathbf{0}.\]
        
        Now, we only need to show that $G$ is the graph corresponding to the Gram matrix of the vectors  $\u_1, \ldots, \u_n$. Since 
        \[ A(\overline{G})_{ij} q_{ij} = \min\{q_{ij}, 0\},\]
        it follows that, for $i\neq j$, 
        \[A(\overline{G})_{ij} = 1 \iff q_{ij}<0.\] 
        As $q_{ij} = \ip{\u_i}{\u_j}$, we conclude that 
        \[ A(G)_{ij} = 1 \iff \ip{\u_i}{\u_j} > 0.\]
        Thus, $G \cong G_{\mathcal{U}}$ where $\mathcal{U} = \{\u_1, \ldots, \u_n\}$.
    \end{proof}

    \begin{claim}[Sufficiency]
        Suppose $G\in \Ger_{\mathcal{V}}(r,n)$. Then $\lambda_k(G) = \alpha_k n-1$.
    \end{claim}

    \begin{proof} Let $G = G_{\mathcal{U}}$ where $\mathcal{U} = \{\u_1, \ldots, \u_n\}$ which arises from the Gerzon ETF $\mathcal{V}$ in $\mathbb{R}^r$ as described in Subsection \ref{subsection:ETF-Gerzon-graphs}. Let $U$ denote the $r\times n$ matrix with columns $\u_1, \ldots, \u_n$. 

    We first argue that $-\alpha_k n $ is an eigenvalue of $\overline{G}$ with multiplicity at least $r$. Note that 
    \[UU^\top 
    = 
    \sum_{i=1}^n \u_i\u_i^\top
    =
    t\sum_{a=1}^N \v_a\v_a^\top
    =
    \frac{tN}{r}I_r
    =
    \frac nr I_r.
\]
In particular, this means that $\rank(U) = r$. So it suffices to show 
\[ A(\overline{G})U^\top = -\alpha_kn \ U^\top.\]
By definition, $A(\overline{G})_{ij} = 1$ if and only if $\ip{\u_i}{\u_j} < 0$. 
So consider a column $u_q$ of $U$. Then 
\begin{align*}
  \sum_{\ip{\u_q}{\u_j} < 0} u_j 
  & = 
    \frac{1}{2}\sum_{j=1}^n
    \left(1-\sign\ip{\u_q}{\u_j}\right)\u_j \\
  & = \frac{1}{2}\left(\sum_{j=1}^n \u_j\right) -\frac{1}{2}\left(\sum_{j=1}^n \sign\ip{\u_q}{\u_j}\u_j\right)\\
  & = -\frac{1}{2}\left(\sum_{j=1}^n \sign\ip{\u_q}{\u_j} \u_j\right)\\
  & = -\frac{t}{2} \left(\sum_{j=1}^N \sign\ip{\u_q}{\v_j} \v_j\right)\quad (\text{where }t=\frac{n}{N})\\
  & = -\frac{t}{2} \left(\u_q\u_q^\top + \sqrt{r+2}\sum_{j\neq q}^N \v_j \v_j^\top\right)\u_q  \quad \left(\text{since }\ip{\u_q}{\v_j} = \frac{\sign \ip{\u_q}{\v_j} }{\sqrt{r+2}}\text{ when }j\neq q\right)\\
 & = -\frac{t}{2} \left(\u_q\u_q^\top + \sqrt{r+2}\left(\frac{N}{r}I_r - \u_q\u_q^\top\right)\right)\u_q  \quad \left(\text{as }\sum_{j=1}^N \v_j \v_j^\top = \frac{N}{r} I_r\right)\\
  &  = -\frac{t}{2} \left(1 + \sqrt{r+2}\cdot \frac{N}{r} - \sqrt{r+2}\right)\u_q\\
  & = -\frac{t}{2} \left(\frac{(r-1)\sqrt{r+2}+2}{2}\right)\u_q\quad \left(\text{since }N = \frac{r(r+1)}{2}\right)\\
  & = - \alpha_k n \u_q.
\end{align*}
We conclude that 
\[ A(\overline{G})U^\top = -\alpha_kn \ U^\top,\]
which implies that $-\alpha_k n$ is an eigenvalue of $A(\overline{G})$ with multiplicity at least $r$. 

Now, since $\sum_{i=1}^n\u_i = \mathbf{0}$, we have \[ JU^\top = \mathbf{0}.\]
Thus, 
\[ A(G)U^\top = (J-I-A(\overline{G}))U^\top = (\alpha_k n -1)U^\top,\]
which implies $\alpha_k n -1$ is an eigenvalue of $A(G)$ with multiplicity at least $r$. 

Since $\sum_{i=1}^n \u_i = \mathbf{0}$, we see that the columns of $U^\top$ are orthogonal to the all-ones vector. By the Perron-Frobenius Theorem, the eigenvector of $G$ corresponding to $\lambda_1(G)$ is non-negative, and therefore it cannot be a column of $U^\top$. This means that there are at least $r+1 = k$ (i.e., the Perron-eigenvector and the columns of $U^\top$) eigenvectors of $G$ which correspond to an eigenvalue which is at least $\alpha_k n - 1$. It follows that 
\[\lambda_k(G)\ge \alpha_k n - 1.\]
Since $\alpha_kn-1$ is an upper bound for the $k$-th eigenvalue of $n$-vertex graphs as established previously, we conclude that 
\[ \lambda_k(G) = \alpha_kn - 1.\qedhere\]
\end{proof}
This completes the proof.
\end{proof}

\section{Equality for $\lambda_k$ when $k\in \{3,4,8, 24\}$}
\label{section:special_k}

In this section, we completely characterize the graphs that achieve the equality in the upper bound for $\lambda_k$ in Theorem \ref{thm:k_eigenvalue_bound} when $k \in \{3, 4, 8, 24\}$. To do so, we need to determine $\Ger_{\mathcal{V}}(k-1,n)$ for any Gerzon ETF $\mathcal{V}$ in $\mathbb{R}^{k-1}$ in light of Theorem \ref{thm:k_eigenvalue_bound}.

\begin{theorem}[\cite{Bannai_Sloane_1981,Delsarte_Goethals_Seidel_1977,
Goethals_Seidel_1975, Seidel_1976}, cf. \cite{Gillespie_2018}]\label{thm:Gerzon_ETF_unique}
For $r\in\{2,3,7,23\}$, there is a unique Gerzon ETF $\mathcal{V}$ up to orthogonal transformations, sign changes, and relabeling.  
\end{theorem}

Thus, in each of these cases, to determine the complete list of Gerzon graphs, it suffices to work with one representative ETF. This is precisely what we do next.

\subsection{Equality for $\lambda_3$: Gerzon graphs of dimension $2$}
\label{sec:small_dimension}

\begin{proof}[Proof of Theorem \ref{thm:3_eigenvalue}] In view of Theorem \ref{thm:k_eigenvalue_bound} and \ref{thm:Gerzon_ETF_unique}, we only need to determine the Gerzon graphs in $\Ger_{\mathcal{V}}(2,n)$ for one Gerzon ETF $\mathcal{V}$ in $\mathbb{R}^2$. So consider a representative Gerzon ETF $\mathcal{V} = \{\v_1, \v_2, \v_3\}$ in $\mathbb{R}^2$ where
\[ \v_1=(1,0),
  \qquad
  \v_2=\left(\frac12,\frac{\sqrt3}{2}\right),
  \qquad
  \v_3=\left(-\frac12,\frac{\sqrt3}{2}\right).\]
It is clear that the base graph $G_{\mathcal{V}_*}\cong C_6$, where $C_6$ denotes a $6$-cycle. It remains to determine the blowup multiplicities. So, consider a Gerzon graph $G_{\mathcal{U}}$ where $\mathcal{U} = \{\u_1, \ldots, \u_n\}$ and  $\u_i \in \{\pm \v_1, \pm \v_2, \pm \v_3\}$. Let $V$ be the matrix with columns $\v_1, \v_2, \v_3$. For $i \in \{1, 2, 3\}$, let $p_i$, $n_i$ and $\d = (d_i)\in \mathbb{Z}^3$ be as defined in Subsection \ref{subsection:ETF-Gerzon-graphs}. We must have 
\[ V\d = \0.\]
Using the expressions for $\v_i$'s, we get
\[
  d_1+\frac{1}{2}d_2-\frac{1}{2}d_3=0,
  \qquad
  d_2+d_3=0.
\]
On solving, we get
\[
  d_1=d_3=-d_2.
\]
This means
\[ p_1 = p_3 = n_2 = \frac{t+d_1}{2}\quad \text{and}\quad n_1 = n_3 = p_2 = \frac{t-d_1}{2}.\] 
It is now clear that $G\cong H_{a,b}$ where $a = p_1$ and $b = t-p_1$. 
\end{proof}

\begin{remark}
Observe that $3K_1$ is the smallest graph and $C_6$ is the smallest connected graph that achieves equality in Theorem \ref{thm:3_eigenvalue}.    
\end{remark}

\subsection{Equality for $\lambda_4$: Gerzon graphs of dimension $3$}

\begin{proof}[Proof of Theorem \ref{thm:4_eigenvalue}] We proceed as in the proof of Theorem \ref{thm:3_eigenvalue}. 

Consider a representative Gerzon ETF $\mathcal{V} = \{\v_1, \ldots, \v_6\} \in \mathbb{R}^3$ such that $\v_i = \frac{\w_i}{\|\w_i\|_2}$ where $\w_i$'s are as given below:
\begin{align*}
  \w_1&=(0,1,\varphi),   & \w_2&=(0,1,-\varphi), & \w_3&=(1,\varphi,0), \\ \w_4&=(1,-\varphi,0),  & \w_5&=(\varphi,0,1),  & \w_6&=(\varphi,0,-1).
\end{align*}
Here $\varphi = \frac{1+\sqrt{5}}{2}$ and note that all $\w_i$'s have the same 2-norm. 

It can be checked that the base graph $G_{\mathcal{V}_*}\cong \mathcal{I}_{12}$. We need to determine the blowup multiplicities. As before, for $i\in \{1, \ldots, 6\}$, we must have 
\[ p_i + n_i = \frac{n}{6} \quad \text{and}\quad \sum_{i=1}^6 d_i\w_i = \sum_{i=1}^6 d_i\v_i = \mathbf{0}.\] Thus
\begin{align*}
  d_3+d_4+\varphi(d_5+d_6)&=0,\\
  d_1+d_2+\varphi(d_3-d_4)&=0,\\
  \varphi(d_1-d_2)+d_5-d_6&=0.
\end{align*}
Since $\varphi$ is irrational and the $d_i$'s are integers, we see that
\[
  d_5+d_6=d_3+d_4=0,
  \qquad
  d_3-d_4=d_1+d_2=0,
  \qquad
  d_1-d_2=d_5-d_6=0.
\]
It follows that
\[
  d_1=d_2=\cdots=d_6=0,
\]
and so $p_i=n_i=\frac{t}{2}$ for every $i$. It is now clear that $G\cong \mathcal{I}^{[q]}$ where $q=\frac{t}{2}$.
\end{proof}

\subsection{Equality for $\lambda_8$: Gerzon graphs of dimension $7$}

For a subset $X\subseteq [8]$, let $\1_X \in \mathbb{R}^8$ denote the indicator vector for $X$. Consider the $7$-dimensional subspace of $\mathbb{R}^8$ orthogonal to $\1_{[8]}$ given by
\[
   \1_{[8]}^\perp := \left\{\x \in\mathbb R^8: \1_{[8]}^\top \x=0\right\}.
\]
Consider the complete graph $K_8$ with $V(K_8)=[8]$ and for each edge $e\in E(K_8)$, define
\[
   \v_e:=\frac{4\one_e-\one_{[8]}}{\sqrt{24}}\in \1_{[8]}^\perp.
\]
Then for any pair $e,f\in E(K_8)$, we have
\begin{equation}\label{eq:T8_adjacency}
\ip{\v_e}{\v_f}
 =
 \begin{cases}
   1 & e=f,\\
   \frac{1}{3} & e\ne f\text{ and }e\cap f\ne\emptyset,\\
  -\frac{1}{3} & e\cap f=\emptyset.
 \end{cases}
\end{equation}
Now, consider a linear isometry $\Psi:\1_{[8]}^\perp \rightarrow \mathbb{R}^7$ which maps an orthogonal basis of $\1_{[8]}^\perp$ to the standard basis of $\mathbb{R}^7$. If $\mathcal{V} =\{\v_e : e\in E(K_8)\}\subset \mathbb{R}^8$, then $\Psi(\mathcal{V})=\{\Psi(\v_e): e\in E(K_8)\}$ forms a Gerzon ETF in $\mathbb{R}^7$ since $|E(K_8)|= \frac{7(7+1)}{2}$.

Thus, to determine the base graph $G_{\Psi(\mathcal{V})_*}$, it suffices to determine $G_{\mathcal{V}_*}$. From \eqref{eq:T8_adjacency}, it is clear that 
\[G_\mathcal{V}\cong G_{\mathcal{-V}}\cong L(K_8),\] where $L(K_8)$ denotes the line graph of $K_8$. Thus, $G_{\mathcal{V}_*}$ is obtained from two copies of $L(K_8)$ where the edges between the two copies are described by \eqref{eq:T8_adjacency}. Therefore, the extremal graphs for $\lambda_8$ are obtained from the above described $G_{\mathcal{V}_*}$ by appropriate uneven blowups using eigen vectors $\d$ corresponding to eigenvalue $-3$ of the Seidel matrix of the complement of $L(K_8)$ as described in Subsection \ref{subsection:ETF-Gerzon-graphs}.

\section*{Acknowledgements}
Bojan Mohar is supported in part by the NSERC Discovery Grant R832714 (Canada), by the ERC Synergy grant (European Union, ERC, KARST, project number 101071836), and by the Research Project N1-0218 of ARIS (Slovenia).

\section*{AI statement}
We acknowledge the use of AI tools during the ideation phase. We declare that the text is not AI-generated.

\bibliographystyle{plain}
\bibliography{references}

@article {Schoenberg_1942,
    AUTHOR = {Schoenberg, I. J.},
     TITLE = {Positive definite functions on spheres},
   JOURNAL = {Duke Math. J.},
  FJOURNAL = {Duke Mathematical Journal},
    VOLUME = {9},
      YEAR = {1942},
     PAGES = {96--108},
      ISSN = {0012-7094,1547-7398},
   MRCLASS = {42.4X},
  MRNUMBER = {5922},
MRREVIEWER = {S.\ Bochner},
       URL = {http://projecteuclid.org.proxy.lib.sfu.ca/euclid.dmj/1077493072},
}

@article {Grunbaum_1960,
    AUTHOR = {Gr\"unbaum, B.},
     TITLE = {Projection constants},
   JOURNAL = {Trans. Amer. Math. Soc.},
  FJOURNAL = {Transactions of the American Mathematical Society},
    VOLUME = {95},
      YEAR = {1960},
     PAGES = {451--465},
      ISSN = {0002-9947,1088-6850},
   MRCLASS = {46.00},
  MRNUMBER = {114110},
MRREVIEWER = {F.\ J.\ Murray},
       DOI = {10.2307/1993567},
}

@book {Stein_Weiss_1971,
    AUTHOR = {Stein, Elias M. and Weiss, Guido},
     TITLE = {Introduction to {F}ourier analysis on {E}uclidean spaces},
    SERIES = {Princeton Mathematical Series},
    VOLUME = {No. 32},
 PUBLISHER = {Princeton University Press, Princeton, NJ},
      YEAR = {1971},
     PAGES = {x+297},
   MRCLASS = {42A92 (31B99 32A99 46F99 47G05)},
  MRNUMBER = {304972},
MRREVIEWER = {Edwin\ Hewitt},
}

@article {Goethals_Seidel_1975,
    AUTHOR = {Goethals, J.-M. and Seidel, J. J.},
     TITLE = {The regular two-graph on {$276$} vertices},
   JOURNAL = {Discrete Math.},
  FJOURNAL = {Discrete Mathematics},
    VOLUME = {12},
      YEAR = {1975},
     PAGES = {143--158},
      ISSN = {0012-365X,1872-681X},
   MRCLASS = {05C25},
  MRNUMBER = {384597},
MRREVIEWER = {Donald\ E.\ Taylor},
       DOI = {10.1016/0012-365X(75)90029-1},
}

@incollection {Seidel_1976,
    AUTHOR = {Seidel, J. J.},
     TITLE = {A survey of two-graphs},
 BOOKTITLE = {Colloquio {I}nternazionale sulle {T}eorie {C}ombinatorie
              ({R}oma, 1973), {T}omo {I}},
     PAGES = {481--511},
 PUBLISHER = {Accad. Naz. Lincei, Rome},
      YEAR = {1976},
   MRCLASS = {05C99 (05B20 50C25)},
  MRNUMBER = {550136},
}

@article {Delsarte_Goethals_Seidel_1977,
    AUTHOR = {Delsarte, P. and Goethals, J. M. and Seidel, J. J.},
     TITLE = {Spherical codes and designs},
   JOURNAL = {Geometriae Dedicata},
  FJOURNAL = {Geometriae Dedicata},
    VOLUME = {6},
      YEAR = {1977},
    NUMBER = {3},
     PAGES = {363--388},
   MRCLASS = {05B99},
  MRNUMBER = {485471},
MRREVIEWER = {Michel\ Deza},
       DOI = {10.1007/bf03187604},
}

@article {Bannai_Sloane_1981,
    AUTHOR = {Bannai, Eiichi and Sloane, N. J. A.},
     TITLE = {Uniqueness of certain spherical codes},
   JOURNAL = {Canadian J. Math.},
  FJOURNAL = {Canadian Journal of Mathematics. Journal Canadien de
              Math\'ematiques},
    VOLUME = {33},
      YEAR = {1981},
    NUMBER = {2},
     PAGES = {437--449},
      ISSN = {0008-414X,1496-4279},
   MRCLASS = {94B25 (05B30 52A45)},
  MRNUMBER = {617634},
MRREVIEWER = {J.-M.\ Goethals},
       DOI = {10.4153/CJM-1981-038-7},
}

@article {Powers_1989,
    AUTHOR = {Powers, David L.},
     TITLE = {Bounds on graph eigenvalues},
   JOURNAL = {Linear Algebra Appl.},
  FJOURNAL = {Linear Algebra and its Applications},
    VOLUME = {117},
      YEAR = {1989},
     PAGES = {1--6},
      ISSN = {0024-3795,1873-1856},
   MRCLASS = {05C50},
  MRNUMBER = {993025},
MRREVIEWER = {Arnold\ Neumaier},
       DOI = {10.1016/0024-3795(89)90541-7},
}

@article {Hong_1988,
    AUTHOR = {Hong, Yuan},
     TITLE = {Bounds of eigenvalues of a graph},
   JOURNAL = {Acta Math. Appl. Sinica (English Ser.)},
  FJOURNAL = {Acta Mathematicae Applicatae Sinica. English Series. Yingyong
              Shuxue Xuebao},
    VOLUME = {4},
      YEAR = {1988},
    NUMBER = {2},
     PAGES = {165--168},
      ISSN = {0168-9673,1618-3932},
   MRCLASS = {05C50 (05C35)},
  MRNUMBER = {961310},
MRREVIEWER = {W.-K.\ Chen},
       DOI = {10.1007/BF02006065},
}

@article {Hong_1993,
    AUTHOR = {Hong, Yuan},
     TITLE = {Bounds of eigenvalues of graphs},
   JOURNAL = {Discrete Math.},
  FJOURNAL = {Discrete Mathematics},
    VOLUME = {123},
      YEAR = {1993},
    NUMBER = {1-3},
     PAGES = {65--74},
      ISSN = {0012-365X,1872-681X},
   MRCLASS = {05C50},
  MRNUMBER = {1256082},
MRREVIEWER = {Robert\ C.\ Brigham},
       DOI = {10.1016/0012-365X(93)90007-G},
}

@article {Waldron_2009,
    AUTHOR = {Waldron, Shayne},
     TITLE = {On the construction of equiangular frames from graphs},
   JOURNAL = {Linear Algebra Appl.},
  FJOURNAL = {Linear Algebra and its Applications},
    VOLUME = {431},
      YEAR = {2009},
    NUMBER = {11},
     PAGES = {2228--2242},
      ISSN = {0024-3795,1873-1856},
   MRCLASS = {42C15},
  MRNUMBER = {2567829},
MRREVIEWER = {Ursula\ Maria\ Molter},
       DOI = {10.1016/j.laa.2009.07.016},
}

@article {Chalmers_Lewicki_2010,
    AUTHOR = {Chalmers, Bruce L. and Lewicki, Grzegorz},
     TITLE = {A proof of the {G}r\"unbaum conjecture},
   JOURNAL = {Studia Math.},
  FJOURNAL = {Studia Mathematica},
    VOLUME = {200},
      YEAR = {2010},
    NUMBER = {2},
     PAGES = {103--129},
      ISSN = {0039-3223,1730-6337},
   MRCLASS = {46A22 (46B20 47A30 47A58)},
  MRNUMBER = {2725896},
MRREVIEWER = {G.\ Schechtman},
       DOI = {10.4064/sm200-2-1},
}

@article{Nikiforov_2015,
    AUTHOR = {Nikiforov, Vladimir},
     TITLE = {Extrema of graph eigenvalues},
   JOURNAL = {Linear Algebra Appl.},
  FJOURNAL = {Linear Algebra and its Applications},
    VOLUME = {482},
      YEAR = {2015},
     PAGES = {158--190},
      ISSN = {0024-3795,1873-1856},
   MRCLASS = {05C50 (05B20)},
  MRNUMBER = {3365272},
MRREVIEWER = {Ravinder\ Kumar},
       DOI = {10.1016/j.laa.2015.05.016},
       URL = {https://doi.org/10.1016/j.laa.2015.05.016},
}

@article {Foucart_Skryzpek_2017,
    AUTHOR = {Foucart, Simon and Skrzypek, Les\l aw},
     TITLE = {On maximal relative projection constants},
   JOURNAL = {J. Math. Anal. Appl.},
  FJOURNAL = {Journal of Mathematical Analysis and Applications},
    VOLUME = {447},
      YEAR = {2017},
    NUMBER = {1},
     PAGES = {309--328},
      ISSN = {0022-247X,1096-0813},
   MRCLASS = {46B20 (42C15)},
  MRNUMBER = {3566474},
MRREVIEWER = {Kazaros\ Kazarian},
       DOI = {10.1016/j.jmaa.2016.09.066},
       URL = {https://doi-org.proxy.lib.sfu.ca/10.1016/j.jmaa.2016.09.066},
}

@article {Basso_2019,
    AUTHOR = {Basso, Giuliano},
     TITLE = {Computation of maximal projection constants},
   JOURNAL = {J. Funct. Anal.},
  FJOURNAL = {Journal of Functional Analysis},
    VOLUME = {277},
      YEAR = {2019},
    NUMBER = {10},
     PAGES = {3560--3585},
      ISSN = {0022-1236,1096-0783},
   MRCLASS = {47A10 (05C50)},
  MRNUMBER = {4001080},
MRREVIEWER = {Hermann\ K\"onig},
       DOI = {10.1016/j.jfa.2019.05.011},
       URL = {https://doi-org.proxy.lib.sfu.ca/10.1016/j.jfa.2019.05.011},
}

@article {Deregowska_Lewandowska_2023,
    AUTHOR = {Der\c{e}gowska, Beata and Lewandowska, Barbara},
     TITLE = {A simple proof of the {G}r\"unbaum conjecture},
   JOURNAL = {J. Funct. Anal.},
  FJOURNAL = {Journal of Functional Analysis},
    VOLUME = {285},
      YEAR = {2023},
    NUMBER = {2},
     PAGES = {Paper No. 109950, 8},
      ISSN = {0022-1236,1096-0783},
   MRCLASS = {41A65 (15A42 41A44 42C15 46B20)},
  MRNUMBER = {4575685},
MRREVIEWER = {Grzegorz\ Lewicki},
       DOI = {10.1016/j.jfa.2023.109950},
       URL = {https://doi-org.proxy.lib.sfu.ca/10.1016/j.jfa.2023.109950},
}

@article {Linz_2023,
    AUTHOR = {Linz, William},
     TITLE = {Improved lower bounds on the extrema of eigenvalues of graphs},
   JOURNAL = {Graphs Combin.},
  FJOURNAL = {Graphs and Combinatorics},
    VOLUME = {39},
      YEAR = {2023},
    NUMBER = {4},
     PAGES = {Paper No. 82, 4},
      ISSN = {0911-0119,1435-5914},
   MRCLASS = {05C50 (05D99)},
  MRNUMBER = {4614672},
       DOI = {10.1007/s00373-023-02678-0},
       URL = {https://doi-org.proxy.lib.sfu.ca/10.1007/s00373-023-02678-0},
}

@article {Leonida_Li_2026,
    AUTHOR = {Leonida, Giacomo and Li, Sida},
     TITLE = {On graphs with large third eigenvalue},
   JOURNAL = {Linear Algebra Appl.},
  FJOURNAL = {Linear Algebra and its Applications},
    VOLUME = {741},
      YEAR = {2026},
     PAGES = {66--96},
      ISSN = {0024-3795,1873-1856},
   MRCLASS = {05C50 (15A18)},
  MRNUMBER = {5059945},
       DOI = {10.1016/j.laa.2026.03.031},
       URL = {https://doi-org.proxy.lib.sfu.ca/10.1016/j.laa.2026.03.031},
}

@misc{Gillespie_2018,
      title={Equiangular lines, Incoherent sets and Quasi-symmetric designs}, 
      author={Neil I. Gillespie},
      year={2018},
      eprint={1809.05739},
      archivePrefix={arXiv},
      primaryClass={math.MG},
      url={https://arxiv.org/abs/1809.05739}, 
}

@misc{Tang_2026,
      title={A sharp upper bound on the third adjacency eigenvalue of a graph}, 
      author={Quanyu Tang},
      year={2026},
      eprint={2603.21181},
      archivePrefix={arXiv},
      primaryClass={math.CO},
      url={https://arxiv.org/abs/2603.21181}, 
}

@misc{Sivashankar_2026,
      title={Upper bound on the $k$-th eigenvalue of a graph}, 
      author={Varun Sivashankar},
      year={2026},
      eprint={2603.28738},
      archivePrefix={arXiv},
      primaryClass={math.CO},
      url={https://arxiv.org/abs/2603.28738}, 
}

@misc{Wakhare_2026,
      title={Graph Eigenvalues and Projection Constants}, 
      author={Tanay Wakhare},
      year={2026},
      eprint={2603.29280},
      archivePrefix={arXiv},
      primaryClass={math.CO},
      url={https://arxiv.org/abs/2603.29280}, 
}

@article {MR2350682,
    AUTHOR = {Sustik, M\'aty\'as A. and Tropp, Joel A. and Dhillon, Inderjit
              S. and Heath, Jr., Robert W.},
     TITLE = {On the existence of equiangular tight frames},
   JOURNAL = {Linear Algebra Appl.},
  FJOURNAL = {Linear Algebra and its Applications},
    VOLUME = {426},
      YEAR = {2007},
    NUMBER = {2-3},
     PAGES = {619--635},
      ISSN = {0024-3795,1873-1856},
   MRCLASS = {15A36 (15A33 15A57)},
  MRNUMBER = {2350682},
MRREVIEWER = {E.\ W.\ Ellers},
       DOI = {10.1016/j.laa.2007.05.043},
       URL = {https://doi.org/10.1016/j.laa.2007.05.043},
}

@article {MR2149656,
    AUTHOR = {Bodmann, Bernhard G. and Paulsen, Vern I.},
     TITLE = {Frames, graphs and erasures},
   JOURNAL = {Linear Algebra Appl.},
  FJOURNAL = {Linear Algebra and its Applications},
    VOLUME = {404},
      YEAR = {2005},
     PAGES = {118--146},
      ISSN = {0024-3795,1873-1856},
   MRCLASS = {42C15 (05B20 05C50 06D22 47N99)},
  MRNUMBER = {2149656},
MRREVIEWER = {Walter\ Schempp},
       DOI = {10.1016/j.laa.2005.02.016},
       URL = {https://doi.org/10.1016/j.laa.2005.02.016},
}

\vspace{0.4cm}

\affl{Hitesh Kumar}{hitesh.kumar.math@gmail.com, hitesh\_kumar@sfu.ca}{Department of Mathematics, Simon Fraser University, Burnaby, Canada}

\affl{Bojan Mohar}{mohar@sfu.ca}{Department of Mathematics, Simon Fraser University, Burnaby, Canada\\On leave from FMF, Department of Mathematics, University of Ljubljana.}

\affl{Seyed Ahmad Mojallal}{seyed\_ahmad\_mojallal@sfu.ca}{Department of Mathematics, Simon Fraser University, Burnaby, BC, Canada}

\affl{Shivaramakrishna Pragada}{shivaramakrishna\_pragada@sfu.ca}{Department of Mathematics, Simon Fraser University, Burnaby, Canada}

\end{document}